\documentclass[oneside,a4paper,11pt,reqno]{amsart}
\usepackage{color}
\usepackage{amssymb,amsmath,amsthm}
\usepackage{graphics,graphicx}
\usepackage{dsfont}
\usepackage[latin1]{inputenc}
\usepackage[T1]{fontenc}

\newtheorem{hypo}{Hypothesis}

\newtheorem{prop}[hypo]{Proposition}

\newtheorem{thm}[hypo]{Theorem}

\newtheorem{lem}[hypo]{Lemma}

\newtheorem{defi}[hypo]{Definition}

\newtheorem{rqe}[hypo]{Remark}

\newtheorem{coro}[hypo]{Corollary}

\newtheorem{exa}[hypo]{Example}

\DeclareMathOperator{\rank}{rank}
\DeclareMathOperator{\class}{class}

\DeclareMathOperator{\val}{val}
\DeclareMathOperator{\sing}{Sing}

\DeclareMathOperator{\Res}{Res}
\DeclareMathOperator{\vect}{Vect}

\def\F{\mathcal{F}}

\title{On the Halphen transform of algebraic space curves}
\date\today

\author{Alfrederic Josse}
\address{Universit\'e de Brest,
Laboratoire de Math\'ematiques de Bretagne Atlantique, UMR CNRS 6205, 29238 Brest cedex, France}
\email{alfrederic.josse@univ-brest.fr}

\author{Fran\c{c}oise P\`ene}
\address{Universit\'e de Brest and Institut Universitaire de France,
Laboratoire de Math\'ematiques de Bretagne Atlantique, UMR CNRS 6205, 29238 Brest cedex, France}
\email{francoise.pene@univ-brest.fr}


\begin{document}
\begin{abstract}
The Halphen transform of a plane curve is the curve obtained by intersecting 
the tangent lines of the curve with the corresponding polar lines with respect
to some conic. This transform has been introduced by Halphen
as a branch desingularization method in \cite{Halphen1} 
and has also been studied in
\cite{Coolidge,fred}. We extend this notion to Halphen transform of a space curve and study several of its properties (birationality, degree, rank, class, desingularization).
\end{abstract}
\maketitle
\section{Introduction}

In \cite{Halphen1}, Halphen studied plane curve transformations based on a simple geometric construction. Given a plane curve $\mathcal C$, choosing a conic $\mathcal K$ in the same plane, every nonsingular point $m$ of $\mathcal C$ is mapped
on the intersection of the tangent line $\mathcal T_m\mathcal C$ to $\mathcal C$ at $m$ with
the polar  of $m$ with respect to $\mathcal K$. The Halphen transform of $\mathcal C$
with respect to $\mathcal K$ is the Zariski closure of the image of $\mathcal C$ by this transformation. 
It is clear that this transformation separates branches at any multiple point of $\mathcal C$ with distinct tangents.
Even more, in \cite{Halphen1}, Halphen shew namely that iterations of these transformations provide a branch desingularization process. 
In \cite{Coolidge}, Coolidge mentionned some properties of Halphen transforms.
Further properties of these transforms have been carefully studied by Josse in \cite{fred}. 

In the present work, we extend the construction of Halphen transforms to curves of the three dimensional complex projective space $\mathbb P^3$.
We show that these transformations can be used for desingularization. We study
the birationality of the Halphen transformation. We establish also formulas
for the degree, the rank and the class of Halphen transforms.
The study of these space Halphen transforms is much more complicated
than the original one for several reasons: these transforms don't act on a hypersurface,
all space curves are not complete intersection of two hypersurfaces, the local  parametrization of their branches is more complicated than for space curves (for which it is 
simply given by a single Puiseux expansion \cite{Halphen,Wall}), etc.

Let $\mathbb P^3:=\mathbb P(\mathbf W)$ where $\mathbf W$ is a four dimensional complex vector space. 
Let $\mathcal C$ be an algebraic curve of $\mathbb P^3$ and 
$\mathcal Q$ be an irreducible quadric of $\mathbb P^3$. 
For $m\in\mathcal C$, we write $\Phi_{\mathcal C,\mathcal Q}(m)$ for the intersection
point of the tangent line $\mathcal T_m\mathcal C$ with the polar plane 
$\delta_m\mathcal Q$ of $\mathcal Q$ with respect to $m\in\mathcal C$
(when this point is well defined).
The {\bf Halphen transform} $\mathcal C^{\mathcal Q}$ of $\mathcal C$
with respect to $ \mathcal Q$
is the Zariski closure of the image of $\mathcal C$ by
$\Phi_{\mathcal C,\mathcal Q}$.
As for plane curves, by definition, this transformation separates branches of nodes of $\mathcal C$.
Recall that the degree of $\mathcal C$ corresponds to the
number of intersection points of $\mathcal C$ with a generic plane
of $\mathbb P^3$, that its rank corresponds to the number
of its tangent lines intersecting a generic line of $\mathbb P^3$ and
that its class is the number of its osculating planes passing through a generic
point of $\mathbb P^3$. We also write $g(\mathcal C)$ for the (geometric)
genus of $\mathcal C$.
\begin{thm}[Numerical characters of Halphen transform]\label{THMDEGRE}
Let $\mathcal C\subset \mathbb P^3$ be an irreducible curve. For a generic
for a generic quadric $\mathcal Q\subset\mathbb P^3$, the Halphen map
$\Phi_{\mathcal C,\mathcal Q}$ is birational, so the Halphen transform preserves the genus, i.e.
\begin{equation}\label{genus}
g(\mathcal C^{\mathcal Q})=g(\mathcal C)\, .
\end{equation}
Moreover, for a generic quadric $\mathcal Q\subset\mathbb P^3$,
the degree, the rank and the class of $\mathcal C^\mathcal Q$ are given by the following formulas
\begin{equation}\label{degre+classe}
\deg\mathcal C^{\mathcal Q}=\deg\mathcal C+\rank \mathcal C\, ,
\end{equation}
\begin{equation}\label{rankHalphen}
\rank\mathcal C^{\mathcal Q}=2(\deg\mathcal C+\rank \mathcal C+g(\mathcal C)-1)-k_0(\mathcal C^\mathcal Q)\,
\end{equation}
and
\begin{equation}\label{classeHalphen}
\class\mathcal C^{\mathcal Q}=3\, \deg\mathcal C+3\, \rank \mathcal C+6\, g(\mathcal C)-6-2\, k_0(\mathcal C^\mathcal Q)-k_1(\mathcal C^\mathcal Q)\, ,
\end{equation}
where $k_i(\mathcal C^\mathcal Q)$ is the $i$-th stationary index of $\mathcal C^\mathcal Q$.
\end{thm}
Let us indicate that \eqref{rankHalphen} and \eqref{classeHalphen} follow directly
from \eqref{genus} and \eqref{degre+classe} thanks to formulas established by 
Piene in \cite{Piene} and recalled in Section \ref{Piene}.
Let us precise that $k_0(\mathcal C)$ corresponds to the number of cusps of $\mathcal C$ (with their multiplicities) and  that $k_1(\mathcal C)$ corresponds to the number of inflection points of $\mathcal C$ (with their multiplicities). The precise definitions of these quantities are recalled in Section \ref{Piene}.
\begin{rqe}
Let $\mathcal C$ be an irreducible algebraic curve of $\mathbb P^3$.

Proposition \ref{desing} below gives the type of the branch $\Phi_{\mathcal C,\mathcal Q}(\mathcal B)$ of $\mathcal C^\mathcal Q$ for any type of branch $\mathcal B$ of $\mathcal C$ for a generic quadric $\mathcal Q\subset\mathbb P^3$. This precise result has several consequences.

First, it enables the computation of the indices $k_0(\mathcal C^\mathcal Q)$ and $k_1(\mathcal C^\mathcal Q)$
for a generic quadric $\mathcal Q\subset\mathbb P^3$, given the type
of the singular branches and of the smooth inflectional branches of $\mathcal C$.

Second, it ensures that the Halphen transform (for a generic quadric $\mathcal Q$)
decreases the order of contact of non singular inflectional branches of $\mathcal C$ with their
tangent lines
and that, except in very special cases, it also decreases the singularities of the singular branches of $\mathcal C$ (see Corollary \ref{desingularisation}).
\end{rqe}

In Section \ref{sec:polar}, we recall some facts on tangent curves, polar surfaces, rank and tangent developable.
In Section \ref{sec:defi}, we detail the construction of the Halphen transform of a space curve.  
In Section \ref{sec:degre}, we prove the degree formula. 
In Section \ref{sec:desing}, we see that the Halphen transform can be used as a branch desingularization method. In Section \ref{sec:rankclass} we prove the rank and class formulas and illustrate them on examples. In Section \ref{sec:birat}, we prove the birationality of the Halphen map $\Phi_{\mathcal C,\mathcal Q}$
for a generic quadric $\mathcal Q$. 

\section{Recalls on tangency and rank}\label{sec:polar}
Let us recall that the tangent plane $\mathcal T_m\mathcal S$ 
to an algebraic surface $\mathcal S=V(\chi)\subset\mathbb P^3$ (with $\chi\in Sym_d\mathbf W^\vee$)
at a nonsingular\footnote{nonsingular means here that
$d\chi(m)\in \mathbf W^\vee$ is non null.} point $m$ of $\mathcal S$ is the plane 
$\mathcal T_m\mathcal S=V(d\chi(\mathbf{m}))\subset\mathbb P^3$, where $\mathbf m\in\mathbf W\setminus\{\mathbf{0}\}$ is a representant of $m\in\mathbb P^3$ and where $d\chi(\mathbf{m})\in\mathbf W^\vee$
is the differential of $\chi$ at $\mathbf{m}$.

Analogously, given an algebraic curve $\mathcal C=\bigcap_{i=1}^I\mathcal S_i$,
where $\mathcal S_i=V(\chi^{(i)})\subset\mathbb P^3$
are surfaces (with $I\ge 2$ and $\chi^{(i)}\in Sym_{d_i}\mathbf W^\vee$), 
the tangent line $\mathcal T_m\mathcal C$ to $\mathcal C$
at a nonsingular\footnote{nonsingular means here that
$\vect(d\chi^{(1)}(m),..., d\chi^{(I)})(m))\subset \mathbf W^\vee$ has dimension 2.}
point $m$ of $\mathcal C$ is the line 
$\mathcal T_{m}\mathcal C=\bigcap_{i=1}^I V(d\chi^{(i)}(\mathbf m))\subset \mathbb P^3$.

In practice, we will use projective coordinates (by fixing
a vector basis $(\mathbf e_1,\mathbf e_2,\mathbf e_3,\mathbf e_4)$ of $\mathbf W$).
We represent each point $m$ of $\mathbb P^3$  
by its coordinates $[x:y:z:t]$ and we write $\mathbf m=(x,y,z,t)\in\mathbf W\setminus\{\mathbf{0}\}$. 
In coordinates, we identify $\chi\in Sym_d\mathbf W^\vee$
with an homogeneous polynomial  $F\in\mathbb C[x,y,z,t]$ of degree $d$ and we write as usual $F_x$, $F_y$, $F_z$
and $F_t$ for its partial derivatives.
We also identify $d\chi(\mathbf m)$ with the gradient
$\nabla F(\mathbf m)=[F_x(\mathbf m):F_y(\mathbf m):F_z(\mathbf m):F_t(\mathbf m)]$.
Let $\mathcal H^\infty:=V(t)$ be the plane at infinity.
\subsection{Tangent curve}
We write $\mathbb G(1,3)$ for the set of projective lines of $\mathbb P^3$.
As usual we write $\mathbb T\mathcal C \subset\mathbb G(1,3)$ for the {\bf tangent curve}
of $\mathcal C$, that is the Zariski closure of the
set of tangent lines to $\mathcal C$  (see for example \cite[p. 188]{Harris}).
We recall that $\mathbb G(1,3)$  is embedded in $\mathbb P^5$
via the Pl\"ucker embedding.  

Up to a linear change of variables, we assume that $\mathcal C\not\subset\mathcal H^\infty$.
If $\mathcal C$ is contained in a curve $\mathcal C_1=V(F,G)\subset \mathbb P^3$ 
with $F,G\in\mathbb C[x,y,z,t]$ homogeneous.
Then, for any $m\in\mathcal C\setminus(\sing(\mathcal C_1)\cup \mathcal H^\infty)$, the tangent line $\mathcal T_{m}\mathcal C=\mathcal T_{m}\mathcal C_1$ is the line of $\mathbb P^3$ passing through the points $m$ and $\mathfrak t_{\infty,\mathcal C_1}(m)\in\mathbb P^3$, where $\mathfrak t_{\infty,\mathcal C_1}(m)$ is the point of $\mathbb P^3$ with coordinates
$$\boldsymbol{\mathfrak t}_{\infty,\mathcal C_1}(\mathbf{m}):=\left(\begin{array}{c}F_y(\mathbf{m})G_z(\mathbf{m})-F_z(\mathbf{m})G_y(\mathbf{m})\\
F_z(\mathbf{m})G_x(\mathbf{m})-F_x(\mathbf{m})G_z(\mathbf{m})\\
F_x(\mathbf{m})G_y(\mathbf{m})-F_y(\mathbf{m})G_x(\mathbf{m})\\
0\end{array}\right)\in\mathbb C^4.$$
Via the Pl\"ucker embedding \cite{Eisenbud-Harris}, $\mathcal T_m\mathcal C$ is identified with
$\lambda_{\mathcal C_1}(m)$ with $\lambda_{\mathcal C_1}:\mathbb P^3\dashrightarrow\mathbb P( \bigwedge\limits^{2}\mathbf W)\cong\mathbb P^5$ given on coordinates by
\footnote{we recall that where, in coordinates, $\bigwedge\limits^{2}(\mathbf{u}\ \mathbf{v})=
    \left(\begin{array}{c}u_1v_2-u_2v_1\\ u_1v_3-u_3v_1\\
    u_1v_4-u_4v_1\\ u_2v_3-u_3v_2\\ u_2v_4-u_4v_2\\u_3v_4-
    u_4v_3\end{array}\right)\in\mathbb C^6$ for any $\mathbf u,\ \mathbf v\in\mathbf W$.}
$$
\boldsymbol{\lambda}_{\mathcal C_1}(\mathbf m)=\left( \bigwedge\limits^{2}(\mathbf m\  {\mathfrak t}_{\infty,\mathcal C_1}(\mathbf{m}))\right)\in\mathbb C^6$$
and so $\mathbb T\mathcal C$ is identified to $\overline{\lambda_{\mathcal C_1}(\mathcal C)}$.
Due to the Euler Formula applied to $F$ and $G$, on $\mathcal C$, we have
\begin{equation}\label{lien}
\boldsymbol{\lambda}_{\mathcal C_1}=-t\left(\begin{array}{c}
   F_zG_t-G_zF_t\\F_tG_y-F_yG_t\\F_yG_z-F_zG_y\\
   F_xG_t-G_xF_t\\F_zG_x-F_xG_z\\F_xG_y-F_yG_x\end{array}\right).
\end{equation}
\subsection{Dual tangent curve}
We write $\mathbb G(1,\mathbb P(\mathbf W^\vee))$ for the set of projective lines
of $\mathbb P(\mathbf W^\vee)$. 
We consider the duality between $\mathbb G(1,3)$ and  $\mathbb G(1,\mathbb P(\mathbf W^\vee))$ which, to $\mathcal L\in\mathbb G(1,3)$, associates
$\mathcal L^*=\{\varphi\in\mathbb P(\mathbf W^\vee)\ :\ \mathcal L\subset V(\varphi)\subset\mathbb P^3\}$. It is then natural to consider the {\bf dual tangent curve} $(\mathbb T\mathcal C)^*=\{\mathcal L^* : \mathcal L\in\mathbb T\mathcal C\}\subset \mathbb G(1,\mathbb P(\mathbf W^\vee))$.
Observe that $(\mathbb T\mathcal C)^*$ is the Zariski closure of
$\{(\mathcal T_m\mathcal C)^*,\ m\in\mathcal C\setminus\sing(\mathcal C)\}$
and that $(\mathcal T_m\mathcal C)^*$ corresponds to the set of projective planes
of $\mathbb P^3$ containing $\mathcal T_m\mathcal C$.
Via the Pl\"ucker embedding,
$(\mathbb T\mathcal C)^*$ is identified with
the Zariski closure of the image of $\mathcal C$ by the rational map
$\vartheta_{\mathcal C_1}:\mathbb P^3\rightarrow\mathbb P(\bigwedge\limits^{2}\mathbf W^\vee)\cong\mathbb P^5$ given
on coordinates by 
$$\boldsymbol{\vartheta}_{\mathcal C_1}(\mathbf m)= \bigwedge\limits^{2}(\nabla F(\mathbf m)\ 
\nabla G(\mathbf m))\in\mathbb C^6.$$
Comparing $\vartheta_{\mathcal C_1}$ with \eqref{lien}, we conclude that
$\mathbb T\mathcal C$ and $(\mathbb T\mathcal C)^*$ are identified via
Pl\"ucker embeddings up to a linear change of coordinates.
In particular, their images via Pl\"ucker embeddings in $\mathbb P^5$ have same degree. This explains why $\mathbb T\mathcal C$ and $(\mathbb T\mathcal C)^*$
are often considered as the same object.
\subsection{Polar surface}
We call {\bf polar surface} from $\beta\in \left(\bigwedge\limits^{2}\mathbf W^\vee\right)^\vee\cong\mathbb C^6$ of the complete intersection curve 
$\mathcal C_1$ the surface
$V\left(\beta\circ\boldsymbol{\vartheta}_{\mathcal C_1}\right)\subset\mathbb P^3$.
This extends to the three dimensional case the notion of polar curves of plane projective curves. 
\begin{lem}\label{specialpolar}
Let $m$ be a non singular point of $\mathcal C_1$. 
Let $\ell\subset\mathbb P^3$ be a projective line containing
$a,b\in\mathbb P^3$ (with $a\ne b$), then $\mathcal T_m\mathcal C$ intersects the line $\ell$ if and only if
\begin{equation}\label{lienppolairedevelopante}
\left\langle \bigwedge\limits^{2}(\nabla F(\mathbf m)\ \nabla G(\mathbf m)),\bigwedge\limits^{2}(a\ b)\right\rangle=0
\end{equation}
in coordinates.\footnote{with the classical notation
$\langle A,B\rangle=\sum_{i=1}^6a_ib_i$ for any $A=(a_1,...,a_6)$
and $B=(b_1,...,b_6)$ in $\mathbb C^6$.}
\end{lem}
\begin{proof}
The fact that  $\mathcal T_m\mathcal C$ intersects the line $\ell$ 
is equivalent to the existence of $(u,v)\in \mathbb C^2\setminus{(0,0)}$ such that $\langle\nabla F(\mathbf m),u\, a+b\, v\rangle=0$ and  $\langle\nabla G(\mathbf m),u\, a+b\, v\rangle=0$, i.e. $\det\left(\begin{array}{cc}
\langle\nabla F(\mathbf m),a\rangle&\langle\nabla F(\mathbf m),b\rangle\\
\langle\nabla G(\mathbf m),a\rangle&\langle\nabla G(\mathbf m),b\rangle\\
\end{array}\right)=0$, which leads to \eqref{lienppolairedevelopante}.
\end{proof}
Observe that the set of $m\in\mathbb P^3$ satisfying \eqref{lienppolairedevelopante} corresponds to 
the polar surfaces from $\beta$ corresponding 
to the Pl\"ucker embedding of $\ell$.
\subsection{Link with the rank}
The rank of $\mathcal C$ is usually described as the degree of the
{\bf tangent developable} surface of $\mathcal C$, i.e. the
Zariski closure of the union of the tangent curves of $\mathcal C$.
We have also the following interpretation of the rank of $\mathcal C$
in terms of the tangent curve.
\begin{lem}
The rank of $\mathcal C$ corresponds to the degree in $\mathbb P^5$ of the Pl\"ucker embedding of $\mathbb T\mathcal C$.
\end{lem}
\begin{proof}
We work with coordinates.
Recall that the image by the Pl\"ucker embedding of 
$\mathbb G(1,3)$ or $\mathbb G(1,\mathbb P(\mathbf W^\vee))$ is $\mathcal K=V(x_1x_6-x_2x_5+x_3x_4)\subset\mathbb P^5$
if we write $[x_1:...:x_6]$ for the coordinates of a point of $\mathbb P^5$.
Due to Lemma \ref{specialpolar} (and the remark following this lemma), it is enough to prove that for a generic $A=[a_6:-a_5:a_4:a_3:-a_2:a_1]\in\mathcal K$, the hyperplane
$\mathcal H_A=V(\langle A,\cdot\rangle)\subset\mathbb P^5$ intersects 
$\tilde{\mathcal C}:=\overline{\vartheta_{\mathcal C_1}(\mathcal C)}\subset\mathbb P^5$
transversally. Indeed, due to the Bezout theorem, this will imply that
$$\# (\ell_A\cap \mathcal C)
  =\#(\mathcal H_A \cap \overline{\vartheta_{\mathcal C_1}(\mathcal C)})
  = \deg \overline{\vartheta_{\mathcal C_1}(\mathcal C)},$$
where $\ell_A$ is the projective line of $\mathbb P^3$ 
corresponding to $A\in\mathcal K$ via Pl\"ucker embedding. 

Observe that this is not obvious since the hyperplane
$\mathcal H_A$ with $A\in\mathcal K$ corresponds to the
tangent hyperplane to $\mathcal K$ at $[a_1:a_2:a_3:a_4:a_5:a_6]$ and so these hyperplanes are very particular.

We assume that $\mathcal H_A\cap\tilde{\mathcal C}\subset
\vartheta_{\mathcal C_1}(\mathcal C)$. This is true 
for a generic $A\in\mathcal K$ since the set of projective hyperplanes intersecting the finite set of points
$\tilde{\mathcal C}\setminus \vartheta_{\mathcal C_1}(\mathcal C)$
has codimension 2 in the projective space of projective hyperplanes of $\mathbb P^5$.

Now let us prove that $\mathcal H_A$ intersects 
$\tilde{\mathcal C}$ transversally in the open set $\{x_i\ne 0\}$ 
for every $i=1,...,6$. Take for example $i=6$ and assume $a_6\ne 0$. We set $a_6=1$ to simplify. In the chart $x_6=1$,
$\tilde C=\Phi(\mathcal C^{(0)})$ and $\mathcal H_A\cap\mathcal K=\Phi(\mathcal H^{(0)}_{(a_2,a_3,a_4,a_5)})$
where $\Phi(x_2,x_3,x_4,x_5)=[x_2x_5-x_3x_4:x_3:x_4:x_5]$,
where $\mathcal C^{(0)}\subset\mathbb C^4$ corresponds to the image of $\mathcal C$ by the projection $[x_1:x_2:x_3:x_4:x_5:1]\mapsto (x_2,x_3,x_4,x_5)$ and where 
$$\mathcal H^{(0)}_{(a_2,a_3,a_4,a_5)}=
   V\left((x_2-a_2)(x_5-a_5)-(x_3-a_3)(x_4-a_4)\right)
         \subset\mathbb C^4.$$
Indeed $\mathcal H_A\cap\mathcal K=V(x_1- a_5x_2+a_4x_3+a_3x_4-a_2x_5+a_1x_6,x_1x_6-x_2x_3+x_4x_5) $ and so
$ \mathcal H_A\cap\mathcal K=\Phi(\mathcal H_A^{(0)})$
with 
$$\mathcal H_A^{(0)}=V(a_5x_2-a_4x_3-a_3x_4+a_2x_5-a_1-(x_2x_3-x_4x_5)\subset\mathbb C^4,$$
but this corresponds exactly to $\mathcal H^{(0)}_{(a_2,a_3,a_4,a_5)}$ (since $a_1=a_2a_5-a_3a_4$).
Now noticing that $\mathcal C^{(0)}$ is a curve of $\mathbb C^4$
and that $\mathcal H^{(0)}_{(a_2,a_3,a_4,a_5)}
=(a_2,a_3,a_4,a_5)+\mathcal H^{(0)}_{(0,0,0,0)}$.
We conclude that, for a generic $A\in\mathcal K$, 
$\mathcal C^{(0)}$ and 
$\mathcal H^{(0)}_{(a_2,a_3,a_4,a_5)}$ have transverse intersection and so $\tilde{\mathcal C}$ and $\mathcal H_A$
have also transverse intersection outside $V(x_6)$
(since the differential $D\Phi(x_2,x_3,x_4,x_5)$ is injective).
\end{proof}
\section{Halphen transform}\label{sec:defi}
\subsection{Definition}
We recall that the polar $\delta_{m_1}(\mathcal Q)$ of an irreducible quadric $\mathcal Q=V(Q)\subset\mathbb P^3$
with respect to a point $m_1\in\mathbb P^3$ is the projective plane
$V(\Delta_{\mathbf{m}_1}Q)\subset\mathbb P^3$,
with the usual notation $\Delta_{\mathbf{m}_1}Q=x_1Q_x+y_1Q_y+z_1Q_z+t_1Q_t$
if $\mathbf{m}_1=(x_1,y_1,z_1,t_1)\in\mathbf W$. Observe that, given $Q$ as above, for a generic $m_1\in\mathbb P^3$,
$\delta_{m_1}(\mathcal Q)$ is 
the projective plane passing through 
the points of $\mathcal Q$ at which the tangent plane to $\mathcal Q$
contains $m_1$.
\begin{defi}
Let $\mathcal Q=V(Q)\subset\mathbb P^3$ be an irreducible quadric and $\mathcal C\subset\mathbb P^3$ be an irreducible curve. The {\bf Halphen transform} $\mathcal{C}^{\mathcal Q}$ of $\mathcal{C}$ with respect to $\mathcal Q$ is
the Zariski closure of the set of intersection points $\Phi_{\mathcal C,\mathcal Q}(m)$ 
of the tangent line $\mathcal T_m\mathcal C$ to $\mathcal{C}$ at $m$ with
the polar $\delta_m\mathcal Q$ of $\mathcal Q$ with respect to $m$, when $m$ varies on $\mathcal{C}$.
\end{defi}
With this definition, $\Phi_{\mathcal C,\mathcal Q}(m)$ is the unique
conjugated point of $m$, with respect to $\mathcal Q$, belonging to $\mathcal T_m\mathcal C$.
Let $\mathcal C_1=V(F,G)\subset \mathbb P^3$ be a complete
intersection curve containing $\mathcal C$,
we extend the definition of $\Phi_{\mathcal C,\mathcal Q}$ into a rational map $\Phi_{\mathcal C_1,\mathcal Q}:\mathbb P^3\dashrightarrow\mathbb P^3$ given by
\begin{equation}\label{defiphi} 
\Phi _{\mathcal C_1,\mathcal Q}[x:y:z:t]=\left[ \bigwedge\limits^{3}\left( 
\nabla F(x,y,z,t) \ 
\nabla G(x,y,z,t) \  
\nabla Q(x,y,z,t)
\right) \right] ,
\end{equation}
with the classical notation $\bigwedge\limits^{3}M$
for any matrix $M\in Mat_{4\times 3}(\mathbb C)$ given by
$$  \bigwedge\limits^{3}M=\left(\begin{array}{c}
    -\det \tilde M^{(1)}\\ \det \tilde M^{(2)}\\
    -\det \tilde M^{(3)}\\ \det \tilde M^{(4)}
     \end{array}\right)\in\mathbb C^4,$$
where $\tilde M^{(i)}$ is the $3\times 3$ matrix
obtained from the matrix $M$ by deleting the $i$-th line.
\begin{rqe}\label{basepoints}
The points $m\in\mathcal C$ for which the right hand side of
\eqref{defiphi} is not well defined are the singular points of $\mathcal C_1$
contained in $\mathcal C$ and the points $m\in\mathcal C\cap\mathcal Q$
such that $\mathcal T_m\mathcal C\subset \mathcal T_m\mathcal Q$.
\end{rqe}
\subsection{Halphen transform of rational curves}
It will be useful to consider the symmetric bilinear form $\mathfrak b_Q(\cdot,\cdot)$
on $\mathbf W$ associated to $Q(\cdot)$
and given on coordinates by $\mathfrak b_Q(\mathbf{m}_1,\mathbf{m}_2)=[Q(\mathbf{m}_1+\mathbf{m}_2)-Q(\mathbf{m}_1-\mathbf{m}_2)]/4$.

As $\Phi _{\mathcal C,\mathcal Q}(m)$ is in $\mathcal T_m\mathcal C_1$, $\Phi _{\mathcal C,\mathcal Q}(m)=[a\cdot \mathbf{m}+b\cdot \boldsymbol{\mathfrak t}_{\infty,\mathcal C_1}(\mathbf m)]$ for some 
$[a:b]\in\mathbb P^1(\mathbb C)$. Now the fact that $\Phi _{\mathcal C,\mathcal Q}(m)$ is in $\delta_m\mathcal Q$ means that
$\mathfrak b_Q(\mathbf{m},\Phi _{\mathcal C,\mathcal Q}(\mathbf{m}))=0$, i.e.
that $a\, \mathfrak b_Q(\mathbf{m})+b\, \mathfrak b_Q(\mathbf{m},\boldsymbol{\mathfrak t}_{\infty,\mathcal C_1}(\mathbf m))=0$, hence we have proved the following remark.
\begin{rqe}
If $\mathcal C$ is contained in a curve $\mathcal C_1=V(F,G)\subset \mathbb P^3$ and if $m\in\mathcal C\setminus(\sing(\mathcal C_1)\cup \mathcal H^\infty)$, then
$$ \Phi _{\mathcal C,\mathcal Q}(m)=[\mathfrak b_Q(\mathbf{m},\boldsymbol{\mathfrak t}_{\infty,\mathcal C_1}(\mathbf m))\cdot \mathbf{m} - Q(\mathbf{m})\cdot \boldsymbol{\mathfrak t}_{\infty,\mathcal C_1}(\mathbf m)].$$
\end{rqe}
In particular, if $u\mapsto[\alpha(u)]$ is a local parametrization of $\mathcal C$
around a generic $m_0\in\mathcal C$,
then a local parametrization of $\mathcal C^{\mathcal Q}$ 
around $\Phi_{\mathcal C,\mathcal Q}(m_0)$ is
given by
\begin{equation}\label{parametrisation}
\Psi_{\mathcal C,\mathcal Q}:u\mapsto \Phi_{\mathcal C,\mathcal Q}(\alpha(u)) = 
           [\mathfrak b_Q(\alpha(u),\alpha'(u))\cdot \alpha(u) 
      - Q(\alpha(u))\cdot \alpha'(u)].
\end{equation}
Observe that Halphen transforms are preserved by linear isomorphisms.
It is also worth noticing that the only fixed points of $(\Phi_{\mathcal C,\mathcal Q})_{|\mathcal C}$ are in $\mathcal C\cap\mathcal Q$.
\section{Degree of the Halphen transform}\label{sec:degre}
\subsection{Proof of the degree formula of Theorem \ref{THMDEGRE}}
The following result will be proved in Section \ref{sec:birat}.
\begin{prop}\label{birationality}
Let $\mathcal C$ be an irreducible curve of $\mathbb P^3$.
Then, for a generic $\mathcal Q$, the map $(\Phi_{\mathcal C,\mathcal Q})_{|\mathcal C}$ is birational.
\end{prop}
As a consequence we obtain the following generic result.
\begin{coro}
If $\mathcal C=V(F,G)$ is an irreducible and smooth algebraic curve of $\mathbb P^3$,
then for a generic quadric $\mathcal Q$, $\deg \mathcal C^\mathcal Q=\deg F\times
\deg G\times(\deg F+\deg G-1)$.
\end{coro}
We will write $i_m(\mathcal A,\mathcal B)$ for the intersection number
of a curve $\mathcal A\subset\mathbb P^3$ and  a surface $\mathcal B\subset\mathbb P^3$ at $m$.
Theorem \ref{THMDEGRE} is a consequence of the following result
involving some polar sufaces \cite{Dolga}.
\begin{prop}\label{thm:formuledegre}
Let $\mathcal C$ be an algebraic curve of $\mathbb P^3$ contained 
in a curve $\mathcal C_1$ which is the complete intersection $V(F,G)$ of two algebraic surfaces.
Assume that $\mathcal C$ is irreducible.
Then, for a generic quadric $\mathcal Q$, the degree of the Halphen transform
of $\mathcal C$ with respect to $\mathcal Q$ is given by
\begin{equation}\label{formuledegre}
\deg\mathcal C^\mathcal Q= \deg\mathcal C\times(\deg F+\deg G-1)
      -\sum_{m\in \mathcal C\cap \sing \mathcal C_1} i_m(\mathcal C, \mathcal P_{\mathcal C_1,B}), 
\end{equation}
for a generic $B\in (\mathbb C^6)^\vee$,
where $\sing \mathcal C_1$ is the set of singular
points of $\mathcal C_1$ and where
$\mathcal P_{\mathcal C_1,B}$ is the {\bf polar surface} of $\mathcal C_1$
given  by 
$\mathcal P_{\mathcal C_1,B}:=B(\bigwedge\limits^{2}(\nabla F\ \nabla G))$.
\end{prop}
\begin{proof}
We use the following classical degree formula
(valid for a generic quadric $\mathcal Q$ and a generic 
$A\in(\mathbb C^4)^\vee$): 
\begin{equation}\label{degreeformula1}
\deg\mathcal C^\mathcal Q= \deg\mathcal C\times(\deg F+\deg G-1)
      -\sum_{m\in\mathcal E} i_m(\mathcal C, \mathcal P_{\mathcal C,\mathcal Q,A}),
\end{equation}
where $\mathcal E$ is the set of $m\in\mathcal C$ for which $\bigwedge^3(\nabla F(\mathbf m)\ \nabla G(\mathbf m)\ \nabla Q(\mathbf m))=\mathbf 0$ and where $\mathcal P_{\mathcal C,\mathcal Q,A}$ is the Halphen polar surface given by 
$\mathcal P_{\mathcal C,\mathcal Q,A}:= V\left(A\left(\bigwedge\limits^{3}( 
\nabla F\ \nabla G\ \nabla Q)\right)\right)\subset\mathbb P^3$.

Observe that, due to Remark \ref{basepoints}, for a generic quadric $\mathcal Q$,
$\mathcal E=\mathcal C\cap \sing \mathcal C_1$. 

Let $m\in\mathcal C\cap\sing\mathcal C_1$.
Let us prove that for generic $\mathcal Q$, $A$ and $B$, we have the equality $i_m(\mathcal C, \mathcal P_{\mathcal C,\mathcal Q,A})=
i_m(\mathcal C, \mathcal P_{\mathcal C_1,B})$.
Without loss of generality we assume that $m[0:0:0:1]$. Let $\mathcal B$
be a branch of $\mathcal C$ at $m$ parametrized by some $\alpha=[\alpha^{(x)}:\alpha^{(y)}:\alpha^{(z)}:\alpha^{(t)}]$ with $\alpha^{(t)}\equiv 1$ of the form
\eqref{parametrisationbranche1}.
For a generic quadric $\mathcal Q$ and a generic $A\in(\mathbb P^3)^\vee$, we have
\begin{multline*}
i_m(\mathcal B,\mathcal P_{\mathcal C,\mathcal Q,A})
    =\min_{i,j,k\in\{x,y,z,t\}\ pairwise\ distinct} \min\left(\val\left(
    \alpha^{(i)}\times 
      [F_{j}G_{k}-F_{k}G_{j}]\circ\alpha\right),\right.\\
           \left. \val\left(\alpha^{(i)}\times[F_{i}G_{k}-F_{k}G_{i}]\circ\alpha
              -  \alpha^{(j)}\times[F_{k}G_{j}-F_{j}G_{k}]\circ\alpha\right)\right).
\end{multline*}
Observe first that 
$$\min_{j,k\in\{x,y,z,t\}\ pairwise\ distinct}\val((F_jG_k-F_kG_j)\circ\alpha)\le 
      i_m(\mathcal B,\mathcal P_{\mathcal C,\mathcal Q,A}).$$
Let us prove that this inequality is indeed an equality.
Since $\alpha^{(t)}\equiv 1$, we observe that 
$$ i_m(\mathcal B,\mathcal P_{\mathcal C,\mathcal Q,A})
     \le  \min_{j,k\in\{x,y,z\}\ pairwise\ distinct} \val((F_{j}G_{k}-F_{k}G_{j})\circ\alpha).$$
Now, if $\val((F_tG_\ell-F_\ell G_t)\circ\alpha) < \min_{j,k\in\{x,y,z\}\ pairwise\ distinct} \val((F_{j}G_{k}-F_{k}G_{j})\circ\alpha)$,
then 
$$ \val \left(\alpha^{(t)}\times[F_{t}G_{\ell}-F_{\ell}G_{t}]\circ\alpha
              -  \alpha^{(j)}\times[F_{\ell}G_{j}-F_{j}G_{\ell}]\circ\alpha\right)
      = \val((F_tG_\ell-F_\ell G_t)\circ\alpha),$$
for any $j\in\{x,y,z\}\setminus\{\ell\}$.
Hence
\begin{eqnarray*}
i_m(\mathcal B,\mathcal P_{\mathcal C,\mathcal Q,A})
       &=&\min_{j,k\in\{x,y,z,t\}\ pairwise\ distinct}\val((F_jG_k-F_kG_j)\circ\alpha)\\
&=&i_m(\mathcal B,\mathcal P_{\mathcal C_1,B}),
\end{eqnarray*}
for a generic $B\in (\mathbb C^6)^\vee$. This ends the proof of Theorem
\ref{thm:formuledegre}.
\end{proof}
Recall that $\deg_{\mathbb P^5}\mathbb T\mathcal C=\deg 
\overline{\tilde\vartheta_{\mathcal C_1}(\mathcal C)}$ with 
$\tilde\vartheta_{\mathcal C_1}:\mathbb P^3\rightarrow\mathbb P^5$
given by $\boldsymbol{\tilde\vartheta}_{\mathcal C_1}(\mathbf m)=\left[\bigwedge\limits^{2}(\nabla F(\mathbf m)\ \nabla G(\mathbf m))\right]\in\mathbb C^6$. 
Since $(\tilde\vartheta_{\mathcal C_1})_{|\mathcal C}$ is birational,
we also have
\begin{equation}\label{degreeformula}
\rank \mathcal C=\deg  \tilde\vartheta_{\mathcal C_1}(\mathcal C)= \deg \mathcal C\, (\deg F+\deg G-2)
      -\sum_{m\in\mathcal C\cap\sing\mathcal C_1} 
    i_m(\mathcal C, \mathcal P_{\mathcal C_1,B}),
\end{equation}
for a generic $B\in(\mathbb C^6)^\vee$.
\begin{proof}[Proof of Theorem \ref{THMDEGRE}]
We combine \eqref{formuledegre} and \eqref{degreeformula}.
\end{proof}
\subsection{Examples}
\begin{exa}
Consider the Viviani curve $\mathcal V= V(x^2+y^2+z^2-t^2,x^2-xt+y^2)$
in $\mathbb P^3$.
In this case $\mathcal C_1=\mathcal C$ and \eqref{degreeformula} becomes
$\rank\mathcal V =8-i_{P}(\mathcal V, \mathcal P_{\mathcal V,B})$
with $P[1:0:0:1]$. Observe that $P$ is a singular point of $\mathcal V$ 
of multiplicity 2. The tangent cone of $\mathcal V$ at $P$ is $V(y^2-z^2,x-t)$.
Moreover the tangent plane to $\mathcal P_{\mathcal V,B}$ at $P$ is
$V((b_1-b_5)y+(b_2+b_6)z)$. Hence $\mathcal P_{\mathcal V,B}$
is always transverse to $\mathcal V$ at $P$ and $i_{P}(\mathcal V,\mathcal P_{\mathcal V,B})=2$
(for a generic $B\in(\mathbb C^6)^\vee$). So $\rank\mathcal V=6$ and $\deg \mathcal V^\mathcal Q=4+6=10$, for a generic quadric $\mathcal Q$.
\end{exa}

\begin{exa}\label{Twistedcubic}
For the twisted cubic curve $\mathcal K=V(y^2-zx,yz-xt,yt-z^2)\subset\mathcal C_1= V(y^2-zx,yz-xt)$
in $\mathbb P^3$ considered in Example \ref{twistedcubic1}, \eqref{degreeformula} becomes
$\rank\mathcal C=6-i_{P}(\mathcal K, \mathcal P_{\mathcal K,B}),$
with $P[0:0:0:1]$. Recall that $\mathcal C_1=\mathcal L\cup\mathcal K$ where
$\mathcal L=V(x,y)$ in $\mathbb P^3$. Observe that $P$ is a non singular point of $\mathcal K$ with tangent line $\mathcal L$ and with osculating plane $V(x)$.  The tangent plane to $\mathcal P_{\mathcal C_1,B}$ at $P$ is $V(2b_1y-b_2x)$
which is tangent but generically not osculating to $\mathcal K$ at $P$.
Hence $i_{P}(\mathcal K,\mathcal P_{\mathcal C_1,B})=2$
(for a generic $B\in(\mathbb C^6)^\vee$). We obtain
$\rank\mathcal K=4$ and 
$\deg \mathcal K^\mathcal Q=3+4=7$, for a generic quadric $\mathcal Q$.
\end{exa}
\subsection{Rational curves}
Consider a rational curve parametrized by a morphism $\gamma:\mathbb P^1\rightarrow\mathbb P^3$. Due to \eqref{parametrisation}, $\mathcal C^\mathcal Q$ is the image of the rational map $\psi:\mathbb P^1
\dashrightarrow \mathbb P^3$ given by $\psi_{\mathcal C,\mathcal Q}:=[\mathfrak b_Q(\gamma,\gamma_u)\cdot\gamma-\mathfrak b_Q(\gamma,\gamma)\cdot \gamma_u ] $. Using twice the Euler formula for $\gamma$,
we obtain that
\begin{equation}\label{psi}
\psi_{\mathcal C,\mathcal Q}=[\mathfrak b_Q(\gamma,\gamma_u)\cdot\gamma_v -\mathfrak b_Q(\gamma,\gamma_v)\cdot \gamma_u ].
\end{equation}
Moreover, via the Pl\"ucker embedding, $\mathbb T\mathcal C$ 
corresponds to the Zariski closure of the image of the morphism $\eta:\mathbb P^1 
\rightarrow \mathbb P^5$ defined on coordinates by
$\eta:=\left[\bigwedge\limits^{2}(\gamma\ \gamma_u)\right]
=\left[\bigwedge\limits^{2}(\gamma_u\ \gamma_v)\right]$ (since $\gamma=u\gamma_u+v\gamma_v$).
Hence we have proved the following result.
\begin{prop}
For a generic space rational curve $\mathcal C$ of degree $d$,
$\rank \mathcal C=2d-2$ and,
for a generic quadric $\mathcal Q$, the degree
of $\mathcal C^\mathcal Q$ is $3d-2$.

This is true for any smooth rational curve image of some morphism 
$\gamma:\mathbb P^1\rightarrow\mathbb P^3$ such that  the coordinates
of $\bigwedge\limits^{2}(\gamma_u\ \gamma_v)$
have no common prime factor.
\end{prop}
\begin{exa}\label{twistedcubic1}
For a generic quadric $\mathcal Q$, the Halphen transform of the twisted 
cubic curve $\mathcal K$ (image of the morphism $\gamma:\mathbb P^1\rightarrow\mathbb P^3$
given in coordinates by $\gamma(u,v)=(u^3,u^2v,uv^2,v^3)$) has degree 7.
Moreover $\rank\mathcal K=4$
(as already obtained in Example \ref{Twistedcubic}).
\end{exa}
\section{Branch desingularization}\label{sec:desing}
\subsection{General case}
Consider a branch $\mathcal B$ of type $(e,r,s)$ of $\mathcal C$
at $m_0$.
Up to a linear change of variables, we assume that $m_0[0:0:0:1]$
and that a parametrization of $ \mathcal B$ is given by
\begin{equation}\label{parametrisationbranche}
\alpha:u\mapsto\left[u^e\sum_{n\ge 0}a_nu^n:u^r\sum_{k\ge 0}b_k u^k:
u^s\sum_{\ell\ge 0}c_\ell u^\ell:1\right],
\end{equation}
with $0<e<r<s$ and $a_0=b_0=c_0=1$. 
Observe that the components of $\alpha$ are in $\mathbb C[[u]]$.
We then 
define 
$$n_0:=\inf\{n\ge 1\ :\ a_n\ne 0\},\quad k_0:=\inf\{k\ge 1\ :\ b_k\ne 0\}\quad\mbox{and}\quad  \ell_0:=\inf\{\ell\ge 1\ :\ c_\ell\ne 0\}.$$
We assume (without loss of generality) that $n_0\ne r$ and that $k_0\ne s$.

The following proposition determines the type
of the branch $\Phi_{\mathcal C,\mathcal Q}(\mathcal B)$ for every branch 
$\mathcal B$ of $\mathcal C$ and for a generic quadric $\mathcal Q$ of $\mathbb P^3$.
\begin{prop}\label{desing}
Let $\mathcal C$ be a curve of $\mathbb P^3$.
Let $\mathcal B$ be a branch of type $(e,r,s)$ of $\mathcal C$
with parametrization \eqref{parametrisationbranche}.

If $r\ne 2e$ and $s\ne 2e$, then for a generic quadric $\mathcal Q$, 
$\Phi_{\mathcal C,\mathcal Q}(\mathcal B)$ is a branch
of type $(a,b,c)$ with $\{a,b,c\}=\{e,r-e,s-e\}$.

If $r=2e$ and if $s\ne\min(\val(\alpha_1^2-\alpha_2),3e)$, then for a generic quadric $\mathcal Q$, 
$\Phi_{\mathcal C,\mathcal Q}(\mathcal B)$ is a branch
of type $(a,b,c)$ with $\{a,b,c\}=\{e,\min(2e,\val(\alpha_1^2-\alpha_2)-e),s-e\}$.

If $r=2e$ and if $s=\val(\alpha_1^2-\alpha_2)<3e$, then for a generic quadric $\mathcal Q$, 
$\Phi_{\mathcal C,\mathcal Q}(\mathcal B)$ is a branch
of type $(e,s-e,\min(2e,\val(\alpha_1^2-\alpha_2-\gamma\alpha_3)-e))$,
with $\gamma:=\sum_{k=0}^sa_ka_{s-k}-b_s$.

If $r=2e$ and if $s=3e\le \val(\alpha_1^2-\alpha_2)$, then for a generic quadric $\mathcal Q$, 
$\Phi_{\mathcal C,\mathcal Q}(\mathcal B)$ is a branch
of type $(e,2e,
\min(3e, \val(\alpha_1^2-\alpha_2-\gamma\alpha_3)-e,\val(\alpha_1\alpha_2-2\alpha_3)-e,\val(\alpha'_1\alpha_2-2\alpha'_2\alpha_1+\alpha'_3)-e+1))$.

If $s=2e$, then for a generic quadric $\mathcal Q$, 
$\Phi_{\mathcal C,\mathcal Q}(\mathcal B)$ is a branch
of type $(r-e,e,\min(\val(\alpha_1^2-\alpha_3)-e,2e))$.
\end{prop}
\begin{coro}[Singularity and inflexion decrease]\label{desingularisation}
We observe that for a generic quadric $\mathcal Q$, $ \Phi_{\mathcal C,\mathcal Q}$ transforms a branch $\mathcal B$ of type $(e,r,s)$
in a branch of type $(e',r',s')$ with $e'\le e$, $r'\le r$ and $s'\le s$
and that $(e',r',s')\ne(e,r,s)$ except if $r=2e$ and $s=3e$ and
$$4e=\val(\alpha_1\alpha_2-2\alpha_3)=
\val(\alpha_1^2-\alpha_2-\gamma\alpha_3)= \val(\alpha'_1\alpha_2-2\alpha'_2\alpha_1+\alpha'_3)+1.$$

In particular, if $\mathcal B$ is an inflectional nonsingular branch of $\mathcal C$ of type $(1,r,s)$, then, for a generic quadric $\mathcal Q\subset\mathbb P^3$, the type of $\Phi_{\mathcal C,\mathcal Q}(\mathcal B)$ is $(1,r-1,s-1)$ (due to the first case in Proposition \ref{desing}).
\end{coro}
\begin{proof}[Proof of Proposition \ref{desing}]
Let us write $\alpha_1,\alpha_2,\alpha_3,\alpha_4$
(resp. $\psi_1,\psi_2,\psi_3,\psi_4$) for the four coordinates of
$\alpha$ given by \eqref{parametrisationbranche} (resp. of $\Psi$ given by \eqref{parametrisation}). 
Recall that $Q(m)={}^tm\cdot M\cdot m$ for some symmetric matrix $M=(m_{i,j})_{i,j}$. We assume that $m_{4,4}=1$ and $m_{1,4}m_ {2,4}m_{3,4}\ne 0$.
With these notations we have 
$$\psi_{i_0}(u)=\sum_{j\ne i_0}A_j(u)
   [\alpha'_j(u)\, \alpha_{i_0}(u) - \alpha_j(u)\, \alpha'_{i_0}(u)],$$
with $A_j(u):=\sum_i m_{i,j}\, \alpha_i(u)$.
Up to a change of variable, a parametrization of $\Phi_{\mathcal C,\mathcal Q}(\mathcal B)$ is given by $[\theta_1:\theta_2:\theta_3:\theta_4]$
with $\theta_1:=m_{1,4}\psi_1+\psi_4+m_{2,4}\psi_2+m_{3,4}\psi_3$,
$\theta_2:=\psi_2$,
$\theta_3:=\psi_3$ and $\theta_4:=\psi_4$. We have
$$
\theta_1=B_{1,4}\, \alpha'_1+B_{2,4}\, \alpha'_2+B_{3,4}\, \alpha'_3
+B_{2,1}(\alpha'_2\alpha_1-\alpha_2\alpha'_1)+B_{3,1}(\alpha'_3\alpha_1-\alpha_3\alpha'_1)+B_{3,2}(\alpha'_3\alpha_2-\alpha_3\alpha'_2),
$$
with $B_{i,j}:=m_{j,4}A_i-m_{i,4}A_j$.
Observe that $\val B_{i,j}=\val \alpha_1=e$ and so $\val\theta_1=2e-1$.
Moreover $\val\theta_2=r-1$, $\val\theta_3=s-1$ and 
$\val \theta_4=e-1$.
Observe that
Then $\val (\theta_1)=\val(\theta_2)=2e-1$ and
$$
\theta_1(u)=(m_{1,1}-m_{1,4}^2)\alpha_1\alpha'_1\\
+(e+r)(m_{1,2}- m_{1,4}m_{2,4})
   u^{r+e-1}+h_1(u)\, ,
$$
$$\theta_2(u)=-\alpha'_2(u)+m_{1,4}(e-2r) u^{e+r-1}+h_2(u)\, ,$$
$$  {\theta_3(u)}=
     -\alpha'_3(u)+m_{1,4}(e-2s) u^{e+s-1}+h_3(u),$$
with  $\val h_1,\val h_2> r+e-1$ and $\val h_3>s+e-1$.
\begin{itemize}
\item If $e$, $r-e$ and $s-e$ are pairwise distinct,
$\Phi_{\mathcal C,\mathcal Q}(\mathcal B)$  is a branch of type  $(a,b,c)$ with
$\{a,b,c\}=\{e,r-e,s-e\}$ and $a<b<c$.
\item 
Assume now that $r=2e$, i.e. $e=r-e<s-e$ and
\begin{eqnarray*}
\tilde\theta_1(u)&:=&\frac{2}{m_{1,1}-m_{1,4}^2}\theta_1(u)+
   \theta_2(u)\\
    & =&(\alpha_1^2-\alpha_2)'(u)
+6e\frac{m_{1,2}-m_{1,4}m_{2,4}}
   {m_{1,1}-m_{1,4}^2}u^{3e-1}+h(u)
\end{eqnarray*}
with $\val h> 3e-1$ and
$$v_1:=\val (2\theta_1+(m_{1,1}-m_{1,4}^2)\theta_2) =\min(\val(\alpha_1^2-\alpha_2)',3e-1).$$
\begin{itemize}
\item If $v_1\ne s-1$, we conclude that 
$\Phi_{\mathcal C,\mathcal Q}(\mathcal B)$  is a branch of type  $(a,b,c)$ with
$\{a,b,c\}=\{e,v_1-e+1,s-e\}$ and $a<b<c$.
\item If $v_1= s-1<3e-1$, then
$$v_2:=\val(\tilde\theta_1+\gamma \theta_3)=\min(\val(\alpha_1^2-\alpha_2-\gamma\alpha_3)-1,3e-1)>s-1$$
(since $e+s-1>3e-1$).
We conclude that 
$\Phi_{\mathcal C,\mathcal Q}(\mathcal B)$  is a branch of type  
$(e,s-e,\min(2e,v_2-e+1))$.
\item If $v_1= s-1=3e-1$, then $\val \tilde\theta_1 =\val\theta_3 =3e-1
> \val\theta_2=2e-1$ and
\begin{eqnarray*}
\tilde\theta_2&:=&\tilde\theta_1+\left[2\frac {m_{1,2}-m_{1,4}m_{2,4}} {m_{1,1}-m_{1,4}^2}+\gamma-m_{1,4}\right]\theta_3\\
    & =&(\alpha_1^2-\alpha_2-\gamma\alpha_3)'
     +2\frac{m_{1,2}-m_{1,4}m_{2,4}}
     {m_{1,1}-m_{1,4}^2}(\alpha'_1\alpha_2+\alpha_1\alpha'_2-\alpha_3')
     +m_{1,4}(\alpha'_1\alpha_2-2\alpha_1\alpha'_2+\alpha'_3)\\
&\ &\ \ \ +\left[2\frac{m_{1,4}m_{1,2}-m_{1,1}m_{2,4}}
     {m_{1,1}-m_{1,4}^2}-m_{1,1}
        \right]\alpha_1(\alpha'_2\alpha_1-\alpha_2\alpha'_1)
+\left[2\frac{m_{2,2}-m_{2,4}^2}
     {m_{1,1}-m_{1,4}^2}-m_{2,4}\right]\alpha_2\alpha'_2\\
&\ &\ \ \  \ \ \ \ \ \  \ \ \ \ 
  +2\frac{m_{1,3}-m_{1,4}m_{3,4}}
     {m_{1,1}-m_{1,4}^2}(\alpha_3\alpha'_1-\alpha_1\alpha'_3)
    +\tilde h(u),
\end{eqnarray*}
with $\val \tilde h\ge 4e$.
Hence 
$$\val \tilde\theta_2=\min( \val(\alpha_1^2-\alpha_2-\gamma\alpha_3)-1,\val(\alpha_1\alpha_2-2\alpha_3)-1,\val(\alpha'_1\alpha_2-2\alpha'_2\alpha_1+\alpha'_3),4e-1).$$
We conclude that 
$\Phi_{\mathcal C,\mathcal Q}(\mathcal B)$  is a branch of type  
$(e,2e, \val\tilde\theta_2-e+1)$.
\end{itemize}
\item 
Assume that $r-e<s-e=e$. Then $\val \theta_2=r-1<\val\theta_1 =\val\theta_3=2e-1$.
We already now that the type of 
$\Phi_{\mathcal C,\mathcal Q}(\mathcal B)$  starts with $r-e$.
As above, we observe that
\begin{eqnarray*} 
v_3&:=&\val (2\theta_1+(m_{1,1}-m_{1,4}^2)\theta_3)\\
   &=& \min(\val(\alpha_1^2-\alpha_3)',3e-1)    >2e-1
\end{eqnarray*}
and we conclude that 
$\Phi_{\mathcal C,\mathcal Q}(\mathcal B)$  is a branch of type  $(r-e,e,v_3-e+1)$.
\end{itemize}
This ends the proof of Proposition \ref{desing}.
\end{proof}
\subsection{Desingularization via iteration of Halphen maps}
Recall that a branch is smooth if it has type $(1,2,3)$.
Observe that, by generic Halphen transforms, a curve of type $(e,r=e+1,s)$
with $1\ne e<r<s$
becomes a curve of type $(1,e,s_1)$ with $s_1\le\min(s,2e)$
($s_1=s-e$ if $s\ne 2e$). 
Moreover, a curve of type $(1,e,s_1)$ becomes a curve 
of type $(1,2,s_1-e+2)$ in $(e-2)$ steps and a
curve of type $(1,2,s_1-e+2)$ becomes a curve 
of type $(1,2,3)$ in $(s_1-e-1)$ steps. Hence it is desingularized in at most 
$s_1+1$ steps.

Some desingularizations by generic Halphen transforms are summarized in the following scheme on which each arrow corresponds to the Halphen transform for a generic
quadric $\mathcal Q$ of $\mathbb P^3$:
$$
\begin{array}{ccccccc}
(4,5,11)& & & & & &  \\
\downarrow & & & & & &  \\
(1,4,7)&(4,5,10)& & & & &  \\
\downarrow &\downarrow & & &  & & \\
(1,3,6)&(1,4,6)& &(4,5,7) & & &  \\
\downarrow &\downarrow & &\downarrow & & &  \\
(1,2,5)&(1,3,5) &(4,5,6) &(1,2,5)&(2,3,4)&(4,5,9)   &  \\
\searrow &\searrow  & \downarrow&\swarrow &\downarrow  &\downarrow  & \\
  & &  (1,2,4)& &(1,2,4)\mbox{ or }(1,2,3)&(1,3,4) &(3,4,5)    \\
 &   & \searrow & & \swarrow  &\swarrow & \swarrow  \\
& &  &(1,2,3) & & & 
\end{array}
$$
\section{Rank and class of the Halphen transform}\label{sec:rankclass}
\subsection{Proof of the formulas of Theorem \ref{THMDEGRE}}\label{Piene}
Let $\mathcal C$ be a non-planar irreducible curve of $\mathbb P^3$.
To understand Piene's formula for the rank of $\mathcal C$, let
us recall the notion of type of  a branch $\mathcal B$ at $m_0\in\mathcal C$.
We say that the branch $\mathcal B$ has type $(e,r,s)$
if, up to a linear change of coordinates, $m_0[0:0:0:1]$ and
$ \mathcal B$ is parametrized by
\begin{equation}\label{parametrisationbranche1}
\alpha:u\mapsto\left[u^e\sum_{n\ge 0}a_nu^n:u^r\sum_{k\ge 0}b_k u^k:
u^s\sum_{\ell\ge 0}c_\ell u^\ell:1\right],
\end{equation}
with $0<e<r<s$ and $a_0=b_0=c_0=1$.
Recall that $e$ is the multiplicity of $\mathcal B$
and that $s$ is the intersection multiplicity of $\mathcal B$ with its osculating plane.
Then the two first stationnary indices of $\mathcal B$ are given by the following formulas
$$k_0(\mathcal B)=e-1\quad\mbox{and}\quad k_1(\mathcal B)=r-e-1.$$
Observe that, if $\mathcal B$ is smooth and ordinary (since in this case $\mathcal B$
has type $(1,2,3)$), then $k_0( \mathcal B)=k_1(\mathcal B)=0$.

The stationnary indices $k_i(\mathcal C)$ is the sum of the
$k_i(\mathcal B)$ over the set of branches $\mathcal B$
of $\mathcal C$. The $0$-th stationnary index $k_0(\mathcal C)$ corresponds to the number of cusps of $\mathcal C$ (computed with their multiplicities).
The first stationnary index $k_1(\mathcal C)$ corresponds to the number of  (possibly singular) inflection points of $\mathcal C$ (computed with their multiplicities).
In \cite[Example (3.2)]{Piene} (see also \cite[Section 2]{Piene2}), Piene established the following formulas for the rank and the class of $\mathcal C$~:
\begin{equation}\label{formulePienerank}
\rank\mathcal C=2\left[ \deg\mathcal C+g(\mathcal C)-1\right]-k_0(\mathcal C),
\end{equation}
\begin{equation}\label{formulePieneclass}
\class\mathcal C=3\left[ \deg\mathcal C+2 g(\mathcal C)-2\right]-2\, k_0(\mathcal C)-k_1(\mathcal C).
\end{equation}
Let $\mathcal Q$ be a generic quadric of $\mathbb P^3$ such that
$\Phi_{\mathcal C,\mathcal Q}$ is birational and such that $\deg \mathcal C^\mathcal Q=\deg \mathcal C+\rank\mathcal C$. Then $g(\mathcal C^\mathcal Q)=g(\mathcal C)$ and so, using Piene's formulas \eqref{formulePienerank} and  \eqref{formulePieneclass}, we obtain
\begin{eqnarray*}
\rank\mathcal C^\mathcal Q&=&2\left[ \deg\mathcal C^\mathcal Q+g(\mathcal C^\mathcal Q)-1\right]-k_0(\mathcal C^\mathcal Q)\\
&=&2\left[ \deg\mathcal C+\rank\mathcal C+g(\mathcal C)-1\right]-k_0(\mathcal C^\mathcal Q)
\end{eqnarray*}
and
\begin{eqnarray*}
\class\mathcal C^\mathcal Q&=&3\left[ \deg\mathcal C^\mathcal Q+2 g(\mathcal C^\mathcal Q)-2\right]-2\, k_0(\mathcal C^\mathcal Q)-k_1(\mathcal C^\mathcal Q)\\
&=&3\left[ \deg\mathcal C+\rank\mathcal C+2 g(\mathcal C)-2\right]-2\, k_0(\mathcal C^\mathcal Q)-k_1(\mathcal C^\mathcal Q).
\end{eqnarray*}
Now we can use Proposition \ref{desing} to compute $k_0(\mathcal C^\mathcal Q)$ and $k_1(\mathcal C^\mathcal Q)$.
\subsection{Application to examples}
\begin{exa}[Halphen transform of a rational sextic]
Consider the sextic curve $\mathcal C=V(zt^2-x^3-xyt-y^2t, zt-xy+y^2)
\subset \mathbb P^3$ intersection of a non singular cubic $\mathcal K$ with a tangential sphere
$\mathcal S$.
Then
$$\deg\mathcal C=6,\ \ g(\mathcal C)=0,\ \rank(\mathcal C)=7,\ \class\mathcal C=6$$
and
$$\deg\mathcal C^\mathcal Q=13,\ \ g(\mathcal C^\mathcal Q)=0,\ \rank(\mathcal C^\mathcal Q)=24,\ \class\mathcal C^\mathcal Q=32.$$
\end{exa}
\begin{proof}
This curve is the image of the map $\gamma:\mathbb P^1\rightarrow\mathbb P^3$
given by 
$$\gamma([s:t])=\left[-t^2s^4: -\frac{\sqrt{2}}2t^3s^3: \frac{\sqrt{2}}2 t^5s-\frac 12t^6: s^6\right]. $$
Hence
$$\deg\mathcal C=6\quad\mbox{and}\quad g(\mathcal C)=0. $$
The curve $\mathcal C$ has exactly two singular points $P_1[0:0:0:1]$ and $P_2[0:0:1:0]$ and no
nonsingular inflection points. 
The curve $\mathcal C$ admits a single branch $\mathcal B_1$ at $P_1$ parametrized by $\tilde\alpha^{(1)}(t)=\gamma([1:t])$. This branch has type $(2,3,5)$. Hence $k_0(\mathcal B_1)=1$ and $k_1(\mathcal B_1)=0$.
Moreover, due to Proposition \ref{desing}, for a generic quadric $\mathcal Q\subset\mathbb P^3$, the type of $\Phi_{\mathcal C,\mathcal Q}(\mathcal B_1)$ is $(1,2,3)$ and so $k_0(\Phi_{\mathcal C,\mathcal Q}(\mathcal B_1))=k_1(\Phi_{\mathcal C,\mathcal Q}(\mathcal B_1))=0$.

Analogously $\mathcal C$ admits a single branch $\mathcal B_2$ at $P_2$ parametrized by $\tilde\alpha^{(2)}(s)=\gamma([s:1])$ which can be rewritten
$$ \tilde\alpha^{(2)}(s)=\left[\frac 2{1-\sqrt 2s}s^4: \frac{\sqrt{2}}{1-\sqrt 2s}s^3: 1:
       \frac{-2 s^6}{1-\sqrt 2s}\right]$$
and so
$$ \tilde\alpha^{(2)}(s)=\left[2\, s^4\sum_{n\ge 0}\left(\sqrt 2s\right)^n: 
      \sqrt{2}s^3\sum_{n\ge 0}\left(\sqrt 2s\right)^n: 1:
       -2 s^6\sum_{n\ge 0}\left(\sqrt 2s\right)^n\right].$$
Up to a linear change of variable $\tilde\alpha^{(2)}$ can be replaced by the following $\alpha^{(2)}$
fitting the assumptions of our Proposition \ref{desing}~:
$$ \alpha^{(2)}(s)=\left[ s^3\sum_{n\ge 0}\left(\sqrt 2s\right)^n: s^4\sum_{n\ge 0}\left(\sqrt 2s\right)^n:  s^6\sum_{n\ge 0}\left(\sqrt 2s\right)^n : 1\right].$$
In particular $\mathcal B_2$ has type $(3,4,6)$. So $k_0(\mathcal B_2)=2$ and $k_1(\mathcal B_2)=0$.
We observe that we are in the case of a branch of type $(e,r,s)=(3,4,6)$ with $s=3e$ and that
$(\alpha^{(1)}_1(s))^2-\alpha^{(1)}_3(s)= \sqrt{2}s^7+...$ so $\val\left((\alpha^{(1)}_1)^2-\alpha^{(1)}_3\right)=7<9=3e$. Therefore, due to Proposition \ref{desing},  $\Phi_{\mathcal C,\mathcal Q}(\mathcal B_2)$ has type $(1,3,7)$ for a generic quadric $\mathcal Q\subset\mathbb P^3$.
In particular $k_0(\Phi_{\mathcal C,\mathcal Q}(\mathcal B_2))=0$ and $k_1(\Phi_{\mathcal C,\mathcal Q}(\mathcal B_2))=1$.
Hence $k_0(\mathcal C)=k_0(\mathcal B_1)+k_0(\mathcal B_2)=3$,
$k_1(\mathcal C)=k_1(\mathcal B_1)+k_1(\mathcal B_2)=0$,
$k_0(\mathcal C^\mathcal Q)=k_0(\Phi_{\mathcal C,\mathcal Q}(\mathcal B_1))+
k_0(\Phi_{\mathcal C,\mathcal Q}(\mathcal B_2))=0$ and 
$k_1(\mathcal C^\mathcal Q)=k_1(\Phi_{\mathcal C,\mathcal Q}(\mathcal B_1))+
k_1(\Phi_{\mathcal C,\mathcal Q}(\mathcal B_2))=1$.
Using Piene's formulas \eqref{formulePienerank} and \eqref{formulePieneclass}, we obtain that
$$\rank\mathcal C=2\left[ \deg\mathcal C+g(\mathcal C)-1\right]-k_0(\mathcal C)
 =2[6+0-1]-3=7$$
and
$$\class\mathcal C= 3\left[ \deg\mathcal C+2 g(\mathcal C)-2\right]-2k_0(\mathcal C)-k_1( \mathcal C)= 3\cdot 4-2\cdot 3-0=6.$$
Due to Theorem \ref{THMDEGRE}, for a generic quadric $\mathcal Q\subset\mathbb P^3$, we have
$$\deg\mathcal C^\mathcal Q=\deg\mathcal C+\rank\mathcal C=6+7=13,\quad
           g(\mathcal C^\mathcal Q)=g(\mathcal C)=0 ,$$
$$\rank\mathcal C^{\mathcal Q}=2(\deg\mathcal C+\rank \mathcal C+g(\mathcal C)-1)-k_0(\mathcal C^\mathcal Q)=2\cdot 12-0=24
$$
and
$$\class\mathcal C^{\mathcal Q}=3\, \deg\mathcal C+3\, \rank \mathcal C+6\, g(\mathcal C)-6-2\, k_0(\mathcal C^\mathcal Q)-k_1(\mathcal C^\mathcal Q)
=3\cdot 6+3\cdot 7-6-1=32.$$
\end{proof}
\begin{exa}[Halphen transform of a non rational sextic]
Consider the sextic curve $\mathcal C=V(x^2\, z+t\, z^2+y^3,x^2+y^2+z^2-2\, z\, t)
\subset \mathbb P^3$ intersection of a cubic $\mathcal K$ having an E6 singularity with a tangential sphere
$\mathcal S$.
Then
$$\deg\mathcal C=6,\ \ g(\mathcal C)=1,\ \rank \mathcal C=10,$$
and
$$\deg\mathcal C^\mathcal Q=16,\ \ g(\mathcal C^\mathcal Q)=1,\ \rank \mathcal C^\mathcal Q=32.$$
\end{exa}
\begin{proof}
The rational E6 cubic $\mathcal K$ is the image of $f:\mathbb P^2\mapsto \mathbb P^3$
given by
$$f([u:v:w])=[-v^2\, w:-u\, v^2:-v^3:v\, w^2+u^3].$$
Observe that $\mathcal C$ is the image of $\mathcal C'$ by $f$ where
$\mathcal C':=V(3\, v\, w^2+u^2\, v+v^3+2\, u^3)\subset\mathbb P^2$.
Observe also that $\mathcal C=V(x^2+y^2+z^2-2zt,3x^2z+y^2z+z^3+2y^3)\subset\mathbb P^3$
(replacing $x^2z+tz^2+y^3$ by $2[x^2z+tz^2+y^3]+z[x^2+y^2+z^2-2zt]$).

Observe that the point $O[0:0:0:1]$ is the only singular point of $\mathcal C$.
At this point, $\mathcal C$ has a single branch $\mathcal B_1$ 
of multiplicity $3$ with tangent line $V(y,z)$ and of type $(3,4,6)$ parametrized by
$\tilde\alpha^{(1)}$ of the following form
$$\tilde\alpha^{(1)}(t):=\left[\frac{\sqrt{2}}{3^{\frac 14}}
    \left(t^3-\frac 1  {2\sqrt{3}}t^5-
\frac 1 {48}t^7+O(t^9)\right):t^4:\frac 1{\sqrt{3}}t^6+\frac 16t^8+O(t^9):1\right].$$
Recall that $\mathcal P_{\mathcal C,B}$ has equation
$$b_1(2xyz-3xy^2)+b_2(2xz^2-4xzt-x^3)-3b_3xz^2+b_4(3y^2z-3y^2t-yx^2-2yzt)-
b_5(3y^2z+yz^2)-b_6(z^3+z^2t+x^2z).$$
Using the local parametrization $\tilde\alpha^{(1)}$, we obtain
$i_O(\mathcal C,\mathcal P_{\mathcal C, B})=8$ (for a generic $B\in(\mathbb C^6)^\vee$) and so, due to \eqref{degreeformula}, we obtain
$$\rank \mathcal C=6(3+2-2)-8=10.$$
Observe that $k_0(\mathcal C)=3-1=2$.
Combining this with the Piene's formula \eqref{formulePienerank} for the rank, we obtain that 
$$g(\mathcal C)=\frac12[\rank(\mathcal C)+k_0(\mathcal C)]+1-\deg\mathcal C=
         \frac 12[10+2]+1-6=1.$$
Hence, due to due to \eqref{THMDEGRE} and \eqref{degre+classe}, we already know that
$$g(\mathcal C^\mathcal Q)=1\quad\mbox{and}\quad\deg\mathcal C^\mathcal Q=6+10=16$$
for a generic quadric $\mathcal Q\subset\mathbb P^3$.
Up to a linear change of coordinates, we identify $\tilde\alpha^{(1)}$ with $\alpha^{(1)}$ 
satisfying the assumptions of our Proposition \ref{desing}:
$$\alpha^{(1)}(t):=\left[t^3-\frac 1  {2\sqrt{3}}t^5-
\frac 1 {48}t^7+O(t^9):t^4:t^6+\frac 1{2\sqrt{3}}t^8+O(t^9):1\right] $$
Observe that we are in the case of a type $(e,r,s)=(3,4,6)$ with $s=2e$ and 
$\val((\alpha^{(1)}_1)^2-\alpha^{(1)}_3)=8<3e$. So the type of the image branch
is $(r-e,e,8-e)=(1,3,5)$.
Hence $k_0(\mathcal C^\mathcal Q)=0$ and so, due to \eqref{rankHalphen}, we obtain
$$\rank\mathcal C^{\mathcal Q}=2(6+10+1-1)-0=32\, .$$
Observe moreover that if $\mathcal B$ is an inflectional nonsingular branch of $\mathcal C$, then its type is $(1,r,s)$ with $2<r<s$ and then the type of $\Phi_{\mathcal C,\mathcal Q}(\mathcal B)$ is $(1,r-1,s-1)$ (first case in Proposition \ref{desing}) which means that $k_1(\Phi_{\mathcal C,\mathcal Q}(\mathcal B))=k_1(\mathcal B)-1$
whereas $k_1(\Phi_{\mathcal C,\mathcal Q}(\mathcal B))=1+k_1(\mathcal B)=1$. 
\end{proof}
\section{Proof of the birationality}\label{sec:birat}
Observe that Proposition \ref{birationality}  is true if $\mathcal C$ is a line
of $\mathbb P^3$. We will assume from now that $\mathcal C$ is not a line.
The proof leads on the following lemmas.
\begin{lem}
Let $\mathcal C$ be an irreducible curve of $\mathbb P^3$.
Let $m_0$ be a nonsingular point of $\mathcal C$.
Then, for a generic $\mathcal Q$, there exist no $m'\in\mathcal C\setminus\{m_0\}$ such that $\Phi_{\mathcal C,\mathcal Q}(m')
=\Phi_{\mathcal C,\mathcal Q}(m_0)$.
\end{lem}
\begin{proof}
Since Halphen transforms are preserved by linear isomorphisms,
we assume without any loss of generality that $m_0[0:0:0:1]$
and that $t_{m_0}\mathcal C[1:0:0:0]$. Recall that $\mathcal Q=V(Q)$
with $Q$ of the form $Q(m)={}^tm\cdot M\cdot m$ for some symmetric matrix $M=(m_{i,j})_{i,j}$ ($M\cdot m$ corresponds to $\nabla Q/2$). Observe that
$\Phi_{\mathcal C,\mathcal Q}(m_0)=[-m_{4,4}:0:0:m_{1,4}]$.
Let $m_1[x_1:y_1:z_1:t_1]$ be a nonsingular point of $\mathcal C$,
$\Phi_{\mathcal C,\mathcal Q}(m_1)
=\Phi_{\mathcal C,\mathcal Q}(m_0)$ is equivalent to
$\Phi_{\mathcal C,\mathcal Q}(m_0)\in\mathcal T_{m_1}\mathcal C$
and $\Phi_{\mathcal C,\mathcal Q}(m_0)\in\delta_{m_1}\mathcal Q$, i.e. to
\begin{equation}\label{EQUA}
\left\{\begin{array}{c} m_{1,4}F_t(m_1)-m_{4,4}F_x(m_1)=0,\ \ \ 
          m_{1,4}G_t(m_1)-m_{4,4}G_x(m_1)=0\\
x_1(m_{1,4}^2-m_{1,1}m_{4,4})+y_1(m_{1,4}\,m_{2,4}-\,m_{1,2}m_{4,4})
+z_1(m_{1,4}\,m_{3,4}-\,m_{1,3}m_{4,4})=0
\end{array}\right.
\end{equation}
if $\mathcal C$ is contained in a curve $\mathcal C_1=V(F,G)$ of $\mathbb P^3$.
Given a quadric $\mathcal Q$, we write $\mathcal H_{\mathcal Q}$
for the set of $[x_1:y_1:z_1:t_1]\in \mathbb P^3$ satisfying the last
line of \eqref{EQUA}.
Observe that, for any $a\in\mathbb C^*$ and any plane $\mathcal H$ containing $m_0$, there exists $M$ such that $m_{1,4}=a$,
$m_{4,4}=1$ and such that $\mathcal H_{\mathcal Q}=\mathcal H$.

We have to prove that, for a generic $M$, no point of $\mathcal C\setminus\{m_0\}$ is solution of \eqref{EQUA}. 
The contrary would mean that for an infinite number of $a\in\mathbb C^*$, there exists an infinity of planes $\mathcal H$
passing through $m_0$ such that $\mathcal C\setminus\{m_0\}$ intersects $\mathcal H\cap V(a\, F_t-F_x,\ a\, G_t-G_x)$.
Since $\mathcal C$ is irreducible, this would imply that 
$\mathcal C\subset V(a\, F_t-F_x,a\, G_t-G_x)$ for an infinity
of $a$ and so that $\mathcal C\subset V(F_t,F_x,G_t,G_x)$.
This would mean that for every $m\in\mathcal C$, $\nabla F(m)$ and $\nabla G(m)$ are contained in $V(x,t)$. But for a generic
$m\in\mathcal C$, $\nabla F(m)$ and $\nabla G(m)$ are not proportional, so $t_m\mathcal C\in V(y,z)$ and so 
$t_m\mathcal C[1:0:0:0]$. Hence $\mathcal C$ would be the line
$V(x,t)$.
\end{proof}
We reinforce this lemma in the following one.
\begin{lem}
Let $\mathcal C$ be an irreducible curve of $\mathbb P^3$.
There exists $N$ such that
for any nonsingular point $m_0$ of $\mathcal C$,
the set $\mathcal B_{m_0}$ of the quadrics $\mathcal Q$ such that $\#\mathcal C\cap(\Phi_{\mathcal C,\mathcal Q}^{-1}(\{m_0\}))>1$ is contained in an hypersurface $\mathcal K_{m_0}$ of the  $\mathbb P^9$ of quadrics of $\mathbb P^3$, this hypersurface has degree less than $N$.
\end{lem}
\begin{proof}
We use the same notations as in the preceding lemma
and we suppose that $F$ and $G$ are irreducible.
Recall that $\mathcal Q=V(Q)$
with $Q$ of the form $Q(m)={}^tm\cdot M\cdot m$ for some symmetric matrix $M=(m_{i,j})_{i,j}$.
Hence we identify the set of quadrics of $\mathbb P^3$ with
$\mathbb P^9$.
Let us consider the following parametrization of $\mathcal H_\mathcal Q$:
\begin{multline*}
\psi(y_1,z_1,t_1):=(y_1(m_{1,2}m_{4,4}-m_{1,4}\,m_{2,4})
+z_1(m_{1,3}m_{4,4}-m_{1,4}\,m_{3,4}),\\
y_1(m_{1,4}^2-m_{1,1}m_{4,4}),z_1(m_{1,4}^2-m_{1,1}m_{4,4}),t_1(m_{1,4}^2-m_{1,1}m_{4,4})).
\end{multline*}
Let us write $K_1:=m_{1,4}F_t\circ\psi-m_{4,4}F_x\circ\psi$,
$K_2:=m_{1,4}G_t\circ\psi-m_{4,4}G_x\circ\psi$,
$K_3:=F\circ\psi$ and $K_4:=G\circ\psi$.
Due to \eqref{EQUA}, $\mathcal B_{m_0}$ 
is contained in the algebraic variety $\mathcal K_{m_0}$ of quadrics
$\mathcal Q$ given by the vanishing of the resultant with respect to $y_1$:
$$\forall i\in\{1,2\},\quad
   \Res_{z_1}\left(z_1^{-1}\Res_{y_1}(K_3,K_4),\Res_{y_1}(K_i,K_ {i+2})\right)=0.$$
Let us show that a generic $\mathcal Q$
is not in $\mathcal K_{m_0}$ (up to a linear change of variables in $(y,z)$).

Since $\mathcal C\ne V(x,y)$, either $\mathcal C\not\subset
V(F_x,F_t)$ or $\mathcal C\not\subset V(G_x,G_t)$.
Assume for example that $\mathcal C\not\subset V(F_x,F_t)$, and so
$V(F)$ is not a plane.
If $V(G)$ is a plane then it is of the form $V(\alpha.x+\beta.t)$.
Hence, for a generic $\mathcal Q$, $V(K_1)$ does not contain a line
and if $V(K_2)$ contains a line then this line has not the form $V(a.z_1+b.t_1)$.

For a generic $\mathcal Q$, 
$\mathcal C\cap\mathcal H_{\mathcal Q}\subset\mathbb P^3$ is finite, so are
$V(K_3,K_4)\subset\mathbb P^2$ 
and $V(\Res_{y_1}(K_3,K_4))\subset\mathbb P^1$.
Moreover, for a generic $\mathcal Q$, $[0:0:1]$ is an ordinary intersection point of $V(K_3)$ with $V(K_4)$ and
$V(K_3, K_4,z_1)=\{[0:0:1]\}$.
Hence $z_1$ divides $\Res_{y_1}(K_3,K_4)$ but not 
$z_1^{-1}\Res_{y_1}(K_3,K_4)$.
The set $\mathcal E_{\mathcal Q}$ 
of $[y_2:z_2:t_2]\in V(K_3)$ such that
$[z_2:t_2]\in V(z_1^{-1}\Res_{y_1}(K_3,K_4))\subset\mathbb P^1$ 
is finite. Since $V(F)\not\subset
V(F_x,F_t)$, 
$ \mathcal E_{\mathcal Q}\cap V(F_x\circ\psi,F_t\circ\psi)=\emptyset$
(up to a linear change of variables in $(y,z)$). Hence for a generic $\mathcal Q$,
$\mathcal E_{\mathcal Q}\cap V(K_1)=\emptyset$
and so 
$\Res_{z_1}(z_1^{-1}
  \Res_{y_1}(K_3,K_4),\Res_{y_1}(K_1,K_3))\ne 0$.
\end{proof}
\begin{proof}[Proof of Proposition \ref{birationality}]
Let us write $\mathcal C_0$ for the set of non singular points
of $\mathcal C$.
Due to the preceding Lemma, the set of quadrics $Q$ such that $\Phi_{\mathcal C,\mathcal Q}$ is not birational is contained in
$$\mathcal K:=\bigcup_{E\subset\mathcal C_0\, :\, \#E<\infty}
   \bigcap_{m\in\mathcal C_0\setminus E}\mathcal K_m,$$
with $\deg \mathcal K_m\le N$.
Now, due to a standard argument 
(see for example \cite{ajfpCRAS}), we conclude that 
either $\mathcal K=\emptyset$ or $\mathcal K$ is contained
in $K_{m_0}$ for some $m_0\in\mathcal C_0$. In any case 
$\mathcal K$ is contained in a subvariety of the set of quadrics
of $\mathbb P^3$.
\end{proof}

\end{document}